\documentclass[oneside,english]{amsart}
\usepackage[T1]{fontenc}
\usepackage[latin9]{inputenc}
\usepackage{mathrsfs}
\usepackage{amsbsy}
\usepackage{amstext}
\usepackage{amsthm}
\usepackage{amssymb}
\usepackage{wasysym}
\usepackage{esint}
\usepackage{calc}
\usepackage{hyperref}
\usepackage[colorinlistoftodos]{todonotes}
\usepackage{footmisc}
\usepackage{colonequals}
\makeatletter

\usepackage{tikz}
\def \no{\nonumber}

\newcommand{\pa}{\partial}

\newtheorem{theorem}{Theorem}[section]
\newtheorem{lem}[theorem]{Lemma}

\newtheorem{cor}{Corollary}[section]

\newtheorem{thm}{\protect\theoremname}
\usepackage{babel}

\makeatother

\providecommand{\theoremname}{Theorem}

\allowdisplaybreaks

\begin{document}
\title[fully nonlinear boundary]{Liouville Theorem for $k-$curvature equation in half space with fully nonlinear boundary condition }
\author{Wei Wei}
\email{wei\_wei@nju.edu.cn}
\address{School of Mathematics, Nanjing University, Nanjing 210093, P.R. China}
\thanks{The author is partially supported by NSFC No. 12201288, No. 12571218, and BK20220755
and the Alexander von Humboldt Foundation. }
\begin{abstract}
We establish the Liouville theorem for positive constant $\sigma_{k}$-curvature equation
in $\mathbb{R}_{+}^{n}$ and  positive constant boundary $\mathcal{B}_{k}^{g}$ curvature equation, where the boundary curvature $\mathcal{B}_{k}^{g}$  is discovered by Sophie Chen \cite{Chen}  from the natural variational functional for $\sigma_{k}(A_{g})$. 
\end{abstract}

\maketitle

\section{Introduction}

Let $(M,g)$ be a smooth compact Riemannian manifold of dimension  $n\ge3$ with boundary $\pa M$.  Denote $$A_{g}=\frac{1}{n-2}(\mathrm{Ric}_{g}-\frac{R_{g}}{2(n-1)}g)$$
	as the  Schouten tensor of $g$,
	where  $\mathrm{Ric}_g$ and $R_g$ are respectively the Ricci tensor
and the scalar curvature of $g$.
	  Let
\begin{equation}\label{expression of sigmak}
\sigma_{k}(A_{g})=\sum_{1\le i_{1}<\cdots<i_{k}\le n}\delta\bigg(\begin{array}{ccc}
i_{1} & \cdots & i_{k}\\
j_{1} & \cdots & j_{k}
\end{array}\bigg)A_{i_{1}}^{j_{1}}\cdots A_{i_{k}}^{j_{k}}.
\end{equation}

In the field of the $\sigma_k$ curvature on closed manifolds,  abundant developments  have been achieved, see \cite{CGY1, CGY2, CGY3, GW, GWDuke, GLWJDG, GWAdv, Li-Li, Li-Li0, STW} etc.

	 We define the second fundamental form at some $P \in \pa M$ by 
	\[
	L_g(X,Y)=-\langle\nabla_{X}\vec{n},Y\rangle
	\]
		for all $X,Y \in T_P(\pa M)$, where $\vec{n}$ is the inward unit normal at $P$, and 
		$$h_{g}=\frac{1}{n-1}\sum_{\alpha}L_g(e_\alpha,e_\alpha)$$   is the mean curvature with $\{e_1,\cdots,e_{n-1}\}$  an orthonormal basis of $T_P(\pa M)$.
 From the variational aspect,
 S. Chen \cite{Chen}  discovered a natural
 boundary curvature $\mathcal{B}_{k}^{g}$ associated to $\sigma_{k}(A_{g})$  as follows:
\begin{equation}
\mathcal{B}_{k}^{g}=\sum_{i=0}^{k-1}C(n,k,i)\sigma_{2k-i-1,i}\left(A_{g}^{\mathbb{\mathrm{T}}},L_{g}\right)\quad n\geq2k,\label{eq:general boundary terms}
\end{equation}
where $A_{g}^{\mathbb{\mathrm{T}}}$ is the tangential part of $A_{g}$
on $\partial M$, $L_{g}$ is the second fundamental form of $\partial M$, $\sigma_{2k-i-1,i}\left(A_{g}^{\mathbb{\mathrm{T}}},L_{g}\right)$
is the mixed symmetric functions defined in \cite{Chen} and $C(n,k,i)=\frac{(2k-i-1)!(n-2k+i)!}{(n-k)!(2k-2i-1)!!i!}$.
On locally conformally flat $2k$-manifolds, the $\mathcal{B}_{k}^{g}$
curvature  is the boundary term in the Chern-Gauss-Bonnet
formula:

\[
\int_{M^{n}}\sigma_{\frac{n}{2}}(A_{g})dv_{g}+\oint_{\partial M}\mathcal{B}_{\frac{n}{2}}^{g}d\sigma_{g}=\frac{(2\pi)^{\frac{n}{2}}}{\left(\frac{n}{2}\right)!}\chi(M^{n},\partial M).
\]

Chang-Chen  \cite{CC}  raised  the question of whether there exists a metric
$g_{u}\in[g]$ to the following boundary value problem 
\begin{equation}
\begin{cases}
\sigma_{k}(A_{g_{u}})=c &\qquad \text{in}\,\,M,\\
\mathcal{B}_{k}^{g_{u}}=0 &\qquad \text{on}\,\,\partial M,
\end{cases}\label{eq:second boundary equation}
\end{equation}
where $c$ is a positive constant. The question is largely open.

Equation (\ref{eq:second boundary equation}) (see for example  \cite{CC,Chen})
is variational, formulated as: 
\[
\mathcal{F}_{k}'[v]=(2k-n)\bigg[\int_{M^{n}}\sigma_{k}(A_{g})vdv_{g}+\oint_{\partial M}\mathcal{B}_{k}^{g}vd\sigma_{g}\bigg],
\]
where 
\[
\mathcal{F}_{k}:=\int_{M^{n}}\sigma_{k}(A_{g})dv_{g}+\oint_{\partial M}\mathcal{B}_{k}^{g}d\sigma_{g}.
\]
On locally conformally flat manifolds of dimension $n\neq2k$, the
critical points of functional $\mathcal{F}_{k}$ in $\{g_{1}:g_{1}\in[g],vol(M,g_{1})=1\}$
for $n>2k$ are solutions of equation (\ref{eq:second boundary equation}),
see \cite{Chen}. For $n=2k,$ a functional whose critical points
are exactly the solutions of (\ref{eq:second boundary equation})
can be found in Proposition 2.3 in \cite{CW18}.

Currently, the only known outcome of Chang-Chen's question is
the case about four-dimensional manifolds with umbilic boundary, in
which it is equivalent to investigating the equation with Neumann boundary
condition. Specifically, as pointed out by S. Chen \cite{Chen}, when $\partial M$
is umbilic and $g\in\Gamma_{k}^{+}$, $\mathcal{B}_{k}^{g}=0$ if
and only if $h_{g}=0.$ Here the well-known $\Gamma_{k}^{+}$-cone
is defined as 
\[
\Gamma_{k}^{+}:=\{g|\sigma_{1}(A_{g})>0,\cdots,\sigma_{k}(A_{g})>0\}.
\]
S. Chen \cite{Chen} proved that on manifolds with umbilic boundary,
if Yamabe invariant $Y(M^{4},\partial M,[g])>0$ and $\int_{M^{4}}\sigma_{2}\left(A_{g}\right)dv_{g}+\oint_{\partial M}\mathcal{B}_{2}^{g}d\sigma_{g}>0$,
then there exists a metric $\hat{g}\in[g]$ such that $\sigma_{2}(\hat{g}^{-1}A_{\hat{g}})$
is constant and $\mathcal{B}_{2}^{\hat{g}}=0$ on $\partial M.$

When the boundary is non-umbilic or umbilic with non-vanishing $\mathcal{B}_{k}^{g}$
curvature,  for $k\ge 3$ , $\mathcal{B}_{k}^{g}$ is a fully
nonlinear equation involving second tangential derivatives and first derivatives
on the boundary, and for  $k=2$, the boundary operator is quasilinear. The complexity is rare in the literature even for the
 uniformly elliptic  equation of second order. 

To comprehend the nature of the boundary curvature $\mathcal{B}_{k}^{g}$,   we investigate the following alternative equation \eqref{eq:second boundary equation-1} on manifolds with umbilic
boundary 
\begin{equation}
\begin{cases}
\sigma_{k}(A_{g})=c, \qquad  g\in\Gamma_{k}^{+} &\qquad \mathrm{~~in~~} M^{n},\\
\mathcal{B}_{k}^{g}=c_{0} &\qquad \mathrm{~~on~~} \partial M,
\end{cases}\label{eq:second boundary equation-1}
\end{equation}
where $c$ and $c_{0}$ are positive constants, and this equation enjoys
variational structure. This non-vanishing $\mathcal{B}_{k}^{g}$ curvature
prominently involves second-order derivatives, as evidenced by the
following explicit expression of $\mathcal{B}_{k}^{g}$ on  umbilic
boundary $\partial M$ in \cite{Chen}:

\begin{equation}\label{eq:boundary equation}
\begin{aligned}
\mathcal{B}_{k}^{g}:= & \frac{(n-1)!}{(n-k)!(2k-1)!!}h_{g}^{2k-1}\\
 & +\sum_{s=1}^{k-1}\frac{(n-1-s)!}{(n-k)!(2k-2s-1)!!}\sigma_{s}\left(A_{g}^{\mathbb{\mathrm{T}}}\right)h_{g}^{2k-2s-1},
 \end{aligned}
\end{equation}
where $h_{g}$ is the mean curvature  of $\partial M$ with respect to $g$
and $A_{g}^{\mathbb{\mathrm{T}}}$ is tangential part of $A_{g}$
on the boundary.

For simplicity, we denote $g_{u}^{-1}A_{g_{u}}$ as $A_{g_{u}}$ and
$g_{u}^{-1}A_{g_{u}}^{\mathbb{\mathrm{T}}}$ as $A_{g_{u}}^{\mathbb{\mathrm{T}}}$
by abuse of notations. We remark that if $g\in\Gamma_{k}^{+}$, then
the linearization of the operator $(\sigma_{k}^{g},\mathcal{B}_{k}^{g})$
is elliptic, see (\ref{eq:linearization operator uner g-u^=00003D00007B4/(n-2)=00003D00007D})
in Section 2, which is the starting point of the whole paper.

Let $\mathbb{R}^n_+=\{(x_1,\cdots, x_n)\in \mathbb{R}^n|x_n>0\}$. 
In this paper we will build a Liouville theorem in $\mathbb{R}^n_+$  for positive constant boundary $\mathcal{B}_{k}^{g}$  curvature as follows: 
\begin{thm}
\label{thm:half space liouville}Given a positive constant $c_{0}$,
let $g_{u}=u^{\frac{4}{n-2}}|dx|^{2}$ in $\mathbb{R}_{+}^{n}$ satisfy
\begin{equation}
\begin{cases}
\sigma_{k}(A_{g_{u}})=2^{k}\binom{n}{k}, \quad g_{u}\in\Gamma_{k}^{+}, & \mathrm{~~in~~}\,\, \mathbb{R}_{+}^{n},\\
\mathcal{B}_{k}^{g_{u}}=c_{0} &\mathrm{~~on~~}\,\,\partial\mathbb{R}_{+}^{n},
\end{cases}\label{eq:main equation}
\end{equation}
where $\sigma_{k}(A_{g_{u}})$ and $\mathcal{B}_{k}^{g_{u}}$ are  defined by (\ref{expression of sigmak}) and (\ref{eq:boundary equation}) respectively.
Assume that $\lim_{x\rightarrow0}u_{0,1}$ exists, where $u_{0,1}(x):=|x|^{2-n}u\left(\frac{x}{|x|^{2}}\right)$.
Then there exist a positive constant $b\in\mathbb{R}_{+}$ and $(\overline{x}^{\prime},\overline{x}_{n})\in\mathbb{R}^{n}$
such that 
\begin{equation}
u(x^{\prime},x_{n})=\left(\frac{\sqrt{b}}{1+b\left|(x^{\prime},x_{n})-(\overline{x}^{\prime},\overline{x}_{n})\right|^{2}}\right)^{(n-2)/2},\label{eq:standard form}
\end{equation}
where $h_{g_{u}}=-\frac{2}{n-2}u_{x_{n}}u^{-\frac{n}{n-2}}=-2\sqrt{b}\overline{x}_{n}>0$
and $A_{g_{u}}^{\mathrm{T}}=2\mathbb{I}_{(n-1)\times(n-1)}$ satisfy

\begin{align*}
\sum_{s=1}^{k-1}\frac{(n-s)!}{(n-k)!(2k-2s-1)!!n}\sigma_{s}\left(2\mathbb{I}_{(n-1)\times(n-1)}\right)[2\sqrt{b}\overline{x}_{n}]^{2k-2s-1}\\
+\frac{(n-1)!}{(n-k)!(2k-1)!!}[2\sqrt{b}\overline{x}_{n}]^{2k-1} & =\frac{n+1-k}{n}c_{0}.
\end{align*}
\end{thm}

From the proof, we actually know that $\lim_{x\rightarrow0}u_{0,1}$
is a positive constant (see Lemma \ref{lem: how far can bar=00003D00007B=00003D00005Clambda=00003D00007D be?}).
When $\mathcal{B}_{k}^{g_{u}}=0$, the boundary condition becomes
$h_{g_{u}}=0$ and the corresponding Liouville theorem has been established
by Li-Li \cite{Li-Li2}, see more general statement in \cite{Li-Li2}.
As the boundary condition is fully nonlinear for $k\ge 3$ and quasilinear for $k=2$, to ensure that the maximum principle and the Hopf lemma work on boundary, we consider the metric with positive
boundary $\mathcal{B}_{k}^{g_{u}}$ curvature in $\Gamma_{k}^{+}$
cone. 
To prove Theorem \ref{thm:half space liouville}, we apply
the key Lemma \ref{lem:lift lemma} in \cite{Li-Li} to $W:=\frac{2}{n-2}\ln u_{0,1}$
on $\partial\mathbb{R}_{+}^{n}\backslash\{0\}$, where $W(x',0)$
is superharmonic on $\partial\mathbb{R}_{+}^{n}$.

$\vspace{0.5pt}$

With the above theorem, we have the following corollary in unit ball $\mathbb{B}_{1}^{n}$ centered at the origin.
\begin{cor}\label{corollary}
Given a positive constant $c_{0}$, let $g_{w}=w^{\frac{4}{n-2}}|dx|^{2}$
in $\mathbb{B}_{1}^{n}$ satisfy 
\[
\begin{cases}
\sigma_{k}(A_{g_{w}})=2^{k}\binom{n}{k}
& \mathrm{in}\quad\mathbb{B}_{1}^{n},\quad g_{w}\in\Gamma_{k}^{+},\\
\mathcal{B}_{k}^{g_{w}}=c_{0} & \mathrm{on}\quad\partial\mathbb{B}_{1}^{n},
\end{cases}
\]
where $\sigma_{k}(A_{g_{w}})$ and $\mathcal{B}_{k}^{g_{w}}$ are  defined by (\ref{expression of sigmak}) and (\ref{eq:boundary equation}) respectively. Then 
\[
w=(\frac{\sqrt{b}}{1+b|x-\bar{x}|^{2}})^{\frac{n-2}{2}},
\]
where $\bar{x}\in\mathbb{R}^{n}$ and $b\in\mathbb{R}^{+}$ satisfy $\mathcal{B}_{k}^{g_{w}}=c_{0}.$ 
\end{cor}

The Liouville theorem is a fundamental tool in studying semilinear and fully nonlinear equations, especially those arising from geometric nonlinear problems. For example,  the Liouville theorem in Euclidean space plays an important role in the existence of solution to $k$-Yamabe problem, see  Chang-Gursky-Yang \cite{CGY1} and Li-Nguyen \cite{LN} etc.
 As applications of Liouville theorem on sphere, various Sobolev inequalities related with  $\sigma_k$ curvature have been established, see Guan-Wang \cite{GWDuke}, Chang-Yang \cite{CY} and Brendle-Viaclovsky \cite{BV}. 

Extensive studies on Liouville Theorems on $\mathbb{R}^{n}$ and $\mathbb{S}^{n}$
have been conducted in the context of the Yamabe-type equation. For
the semilinear elliptic equations, the Liouville theorem can be traced back
to Obata \cite{Obata}, Gidas-Ni-Nirenberg \cite{GNL} and Caffarelli-Gidas-Spruck
\cite{CGS}. These influential works originate the application of the method of moving
planes/spheres to Liouville theorems. For more details of the method of
moving spheres in  Li-Zhu \cite{LZhu}, Li-Zhang
\cite{Li-Zhang} and references therein. For fully nonlinear equations, especially the $\sigma_{k}$-curvature
equation on $\mathbb{R}^{n}$, Viaclovsky \cite{V1,V2} obtained the
Liouville theorem under the additional hypothesis that $|x|^{2-n}u(x/|x|^{2})$
can be extended to a positive $C^{2}$ function near $x=0$ for $2\le k\le n.$
Concerning $k=2$, Chang-Gursky-Yang \cite{CGY3} utilized Obata's
argument to establish the case $n=4,5$ and higher dimensional
cases under some additional assumptions. For $n=4$, by constructing
a monotone formula with respect to level set of the solution, the
author joint with Fang and Ma \cite{FMW} introduced an alternative approach
and proved a Liouville theorem for some  general $\sigma_{2}$-curvature type
equation, which may not be conformal invariant. Some general cases
for conformally invariant equations including the $\sigma_{k}$-curvature
equation were established by Li-Li \cite{Li-Li0,Li-Li} and Li-Lu-Lu
\cite{LLL}. Also Chu-Li-Li \cite{CLL} derived the necessary and sufficient
conditions for the validity of Liouville-type theorems in $\mathbb{R}^{n}$. 

Concerning the constant boundary mean curvature related to ball, 
Escobar
\cite{Es} classified the metric with constant scalar curvature and
the constant boundary mean curvature.  In $\mathbb{R}^{n}_+$, Liouville theorems for metrics with constant scalar
curvature and constant boundary mean curvature were established by Li-Zhu \cite{LZhu}, Chipot-Shafrir-Fila
\cite{CSF1996} and Li-Zhang \cite{Li-Zhang}. For fully nonlinear
cases including $\sigma_{k}$ curvature in $\mathbb{R}^{n}_+$, Li-Li \cite{Li-Li2}
obtained the corresponding Liouville theorem with constant boundary
mean curvature.

For constant boundary $\mathcal{B}_{k}^{g}$ curvature on  $\partial \mathbb{R}^{n}_+$  or  $\partial \mathbb{S}^{n}_+$, the Liouville theorems  for metrics with constant $\sigma_{k}$-curvature are few as the boundary $\mathcal{B}_{k}^{g}$ curvature brings new difficulties. Assuming that
$\sup_{\mathbb{S}^{n}}h_{g}\le(k+1)\inf_{\mathbb{S}^{n}}h_{g}$,  Case-Wang \cite{CGS} demonstrated an Obata-type theorem
for $\sigma_{k}(A_{g})=0$ on $\mathbb{S}_{+}^{n+1}$ and $\mathcal{B}_{k}^{g}=c_0\in \mathbb{R}_+$
on $\partial\mathbb{S}_{+}^{n+1}$.
Case-Wang \cite{CW2020} classified the local minimizers of $\mathcal{F}_{2}$ among all conformally flat metrics in $\mathbb{B}_{1}^{5}$ and $\mathbb{B}_{1}^{6}$ with unit boundary volume, as well as the local minimizer of a analogous functional in $\mathbb{B}_{1}^{4}$. Furthermore, both authors with Moreira \cite{CMW} obtained
some non-uniqueness results in different setting. 

This paper is concerned with metrics of positive constant $\sigma_{k}$ curvature and positive constant boundary $\mathcal{B}_{k}^{g}$ curvature,   in which the method
of moving spheres works,  since the linearized equation of $(\sigma_{k}^{g},\mathcal{B}_{k}^{g})$
is elliptic.

We organize the paper as follows. In Section 2 we provide some facts
about the linearized operator of $\mathcal{B}_{k}^{g}$. In Section
3, inspired by Li-Li \cite{Li-Li2}, we prove Theorem \ref{thm:half space liouville-1}
in stronger assumption (\ref{eq:strong assmption}) by the method
of moving spheres. In Section 4, due to the super-harmonicity of $\frac{2}{n-2}\ln u_{0,1}$
on  lower dimensional space, we utilize the key lemma in \cite{Li-Li}
in the lower dimensional space, and then obtain Theorem \ref{thm:half space liouville} by  Theorem \ref{thm:half space liouville-1}. 
In Appendix we list some useful lemmas in \cite{Li-Zhang,Li-Li,Li-Li2}
for readers' convenience.

During the paper was submitted, B. Z. Chu, Y. Y. Li and Z. Y. Li \cite{CLL2}
have extended my result by removing the assumption that $\lim_{x\rightarrow0}u_{0,1}$
exists.

\bigskip

\noindent{\bf Acknowledgments.} The author would like to thank
 Prof. X. Z. Chen for enlightening discussions and constant support, and Dr. Biao Ma for his suggestion in writing the Appendix. The author would like to thank the referees for their helpful comments and suggestions, which have improved the presentation and readability of the paper.
 
  \section{Preliminary}

In this section we describe the linearization of $(\sigma_{k}^{g},\mathcal{B}_{k}^{g})$
on $\partial\mathbb{R}_{+}^{n}$, which is an elliptic operator on the boundary depending on the mean curvature of $g$ and the cone condition.

In $\mathbb{R}_{+}^{n}$, if $g_{u}=u^{\frac{4}{n-2}}g_{\mathbb{E}}$,
then 
\begin{equation}\label{mean curvature expression}
h_{g_{u}}=-\frac{2}{n-2}u^{-\frac{n}{n-2}}u_{n}\quad\text{on}\,\,\partial\mathbb{R}_{+}^{n},
\end{equation}
where $u_{n}=\frac{\partial u}{\partial x_{n}}$ $\text{on}\,\,\partial\mathbb{R}_{+}^{n}.$
It holds that

\begin{align*}
g_{u}^{-1}A_{g_{u}}= & -\frac{2}{n-2}u^{-(n+2)/(n-2)}\nabla^{2}u\\
 & +\frac{2n}{(n-2)^{2}}u^{-2n/(n-2)}\nabla u\otimes\nabla u-\frac{2}{(n-2)^{2}}u^{-2n/(n-2)}|\nabla u|^{2}\mathbb{I}_{n\times n},
\end{align*}
and on $\partial\mathbb{R}_{+}^{n}$,
\begin{align}\label{tangential expression}
g_{u}^{-1}A_{g_{u}}^{\mathbb{\mathrm{T}}}= & \bigg[-\frac{2}{n-2}u^{-(n+2)/(n-2)}\nabla^{2}u+\frac{2n}{(n-2)^{2}}u^{-2n/(n-2)}\nabla u\otimes\nabla u\bigg]^{\mathbb{\mathrm{T}}}\no\\
 & -\frac{2}{(n-2)^{2}}u^{-2n/(n-2)}|\nabla u|^{2}\mathbb{I}_{(n-1)\times(n-1)}.
\end{align}

\subsection{Linearization of $\mathcal{B}_{k}^{g}$ on \textmd{$\partial\mathbb{R}_{+}^{n}$}}

We first study the linearization of $\mathcal{B}_{k}^{g_{u}}$. In
\cite{CGS} Case-Wang computed the conformal linearization of $\mathcal{B}_{k}^{g}$
and they used the symbol $H_{k}$ instead. Our notations are consistent
with S. Chen \cite{Chen}. Throughout the paper, $\alpha,\beta,\gamma$ range
from $1,\cdots,n-1$.

 Now we present some computations of the boundary operator, which will
be frequently used.

Here we call operator $(\sigma_{k}^{g_{u}},\mathcal{B}_{k}^{g_{u}})$
is elliptic on $\partial\mathbb{R}_{+}^{n}$, if the matrix $\{\tilde{a}_{ij}\}_{n\times n}$  in the linearization $\mathbb{L}\phi(x)=-\tilde{a}_{ij}\phi_{ij}(x)+\tilde{b}_i\phi_i(x)+\tilde{c}(x)\phi(x)$ of  $\sigma_{k}^{g_{u}}$ is non-negative near $\partial\mathbb{R}_{+}^{n}$, and in the linearized operator  $L\phi=-a_{\alpha\beta}\phi_{\alpha\beta}+b_\alpha \phi_\alpha-b_n\phi_n+c\phi$  of  $\mathcal{B}_{k}^{g_{u}}$, the matrix $\{a_{\alpha\beta}\}_{(n-1)\times (n-1)}$ and the coefficient  $b_n$ are both non-negative on  $\partial\mathbb{R}_{+}^{n}$. And if $\{\tilde{a}_{ij}\}$, $\{a_{\alpha\beta}\}$ and $b_n$ are positive, then we can say the linearization  of $(\sigma_{k}^{g_{u}},\mathcal{B}_{k}^{g_{u}})$ are strictly elliptic. The consistent sign of $\{\tilde{a}_{ij}\}$, $\{a_{\alpha\beta}\}$ and $b_n$ are important in this paper and will be frequently used.  The intention to call such way is for the validity of  maximum principle and the Hopf lemma.  

Since the powers of $h_{g_u}$ in \eqref{eq:boundary equation} are always odd, together with $\mathcal{B}_{k}^{g_{u}}>0$ and $g_{u}\in\Gamma_{k}^{+}$, $h_{g_{u}}$ is always positive if the boundary is umbilic.  This observation can be traced back to \cite{Chen}.

\begin{lem}\label{elliptic}
Let $g_{u}=u^{\frac{4}{n-2}}g_{\mathbb{E}}$. 
Assume that $g_{u}\in\Gamma_{k}^{+}$ and on $\partial\mathbb{R}_{+}^{n}$,  $\mathcal{B}_{k}^{g_{u}}$ 
is positive. Then, operator $(\sigma_{k}^{g_{u}},\mathcal{B}_{k}^{g_{u}})$
is elliptic on $\partial\mathbb{R}_{+}^{n}$.
\end{lem}

\begin{proof}

Assume that $g_{0}=u_{0}^{\frac{4}{n-2}}g_{\mathbb{E}}$
and $g_{1}=u_{1}^{\frac{4}{n-2}}g_{\mathbb{E}}$ satisfy $\mathcal{B}_{k}^{g_{0}}=\mathcal{B}_{k}^{g_{1}}=c_{0}.$
Then, taking $\psi:=u_{1}-u_{0}$ and denote $u=tu_{1}+(1-t)u_{0}$
and $g_{u}=u^{\frac{4}{n-2}}g_{\mathbb{E}},$ 
, we {\bf claim:}
\begin{align}
0 & =\mathcal{B}_{k}^{g_{1}}-\mathcal{B}_{k}^{g_{0}}\nonumber \\
 & =-a_{\alpha\beta}\psi_{\alpha\beta}+b_{\alpha}\psi_{\alpha}-b_{n}\psi_{n}+c\psi:=L\psi,\label{eq:linearization operator uner g-u^=00003D00007B4/(n-2)=00003D00007D}
\end{align}
where $b_{\alpha}$ and $c$ depend on $u_{1},u_{0}$, \begin{equation}
a_{\alpha\beta}:=\int_{0}^{1}\frac{2}{n-2}u^{-\frac{n+2}{n-2}}\sum_{s=1}^{k-1}\frac{(n-1-s)!}{(n-k)!(2k-2s-1)!!}h_{g_{u}}^{2k-2s-1}\frac{\partial\sigma_{s}(A_{g_{u}}^{\mathbb{\mathrm{T}}})}{\partial A_{g_u,\alpha\beta}^{\mathbb{\mathrm{T}}}}dt\label{eq:coeffients of tangential second derivative}
\end{equation}
and 
\begin{equation}
b_{n}:=\frac{2}{n-2}\int_{0}^{1}u^{-\frac{n}{n-2}}\sigma_{k-1}(A_{g_{u}}^{\mathbb{\mathrm{T}}})dt.\label{eq:coefficients of first derivative}
\end{equation}

To prove this {\bf Claim},  for simplicity, denote $C_{1}(n,k,s)=\frac{(n-1-s)!}{(n-k)!(2k-2s-1)!!}$.

Observe that 
\[
0=\mathcal{B}_{k}^{g_{1}}-\mathcal{B}_{k}^{g_{0}}=\int_{0}^{1}\frac{\partial}{\partial t}\big(\mathcal{B}_{k}^{g_{u}}\big)dt,
\]
where $g_{u}=(tu_{1}+(1-t)u_{0})^{\frac{4}{n-2}}g_{\mathbb{E}}$.

Define 
\[
b_{n}^{*}:=\frac{2}{n-2}u^{-\frac{n}{n-2}}\sigma_{k-1}\left(A_{g_{u}}^{\mathbb{\mathrm{T}}}\right)
\]
and 
\[
a_{\alpha\beta}^{*}:=\frac{2}{n-2}u^{-\frac{n+2}{n-2}}\sum_{s=1}^{k-1}C_{1}(n,k,s)h_{g_{u}}^{2k-2s-1}\frac{\partial\sigma_{s}(A_{g_{u}}^{\mathbb{\mathrm{T}}})}{\partial A_{u,\alpha\beta}^{\mathbb{\mathrm{T}}}}.
\]
As $\frac{du}{dt}=\psi=u_{1}-u_{0}$, 
with  (\ref{eq:boundary equation}), we have 
\begin{align*}
 & \text{\ensuremath{\frac{\partial}{\partial t}}(\ensuremath{\mathcal{B}_{k}^{g_{u}}})}\\
= & C_{1}(n,k,0)(2k-1)h_{g_{u}}^{2k-2}\frac{\partial h_{g_{u}}}{\partial t}\\
 & +\sum_{s=1}^{k-1}C_{1}(n,k,s)\frac{\partial}{\partial t}\sigma_{s}\left(A_{g_{u}}^{\mathbb{\mathrm{T}}}\right)h_{g_{u}}^{2k-2s-1}\\
 & +\sum_{s=1}^{k-1}C_{1}(n,k,s)(2k-2s-1)\sigma_{s}\left(A_{g_{u}}^{\mathbb{\mathrm{T}}}\right)h_{g_{u}}^{2k-2s-2}\frac{\partial h_{g_{u}}}{\partial t}\\
= & C_{1}(n,k,0)(2k-1)h_{g_{u}}^{2k-2}\left(\frac{2}{n-2}\frac{n}{n-2}u^{-\frac{2n-2}{n-2}}u_{n}\psi-\frac{2}{n-2}u^{-\frac{n}{n-2}}\psi_{n}\right)\\
 & +\sum_{s=1}^{k-1}C_{1}(n,k,s)h_{g_{u}}^{2k-2s-1}\frac{\partial}{\partial A_{u,\alpha\beta}^{\mathbb{\mathrm{T}}}}\sigma_{s}\left(A_{g_{u}}^{\mathbb{\mathrm{T}}}\right)\bigg\{\frac{2}{n-2}\frac{n+2}{n-2}u^{-\frac{2n}{n-2}}u_{\alpha\beta}\psi\\
 & -\frac{2}{n-2}u^{-\frac{n+2}{n-2}}\psi_{\alpha\beta}-\frac{2n}{(n-2)^{2}}\frac{2n}{n-2}u^{-\frac{2n}{n-2}-1}u_{\alpha}u_{\beta}\psi\\
 & +\frac{2n}{(n-2)^{2}}u^{-\frac{2n}{n-2}}\psi_{\alpha}u_{\beta}+\frac{2n}{(n-2)^{2}}u^{-\frac{2n}{n-2}}\psi_{\beta}u_{\alpha}\\
 & +\frac{2}{(n-2)^{2}}\frac{2n}{n-2}u^{-\frac{2n}{n-2}-1}|\nabla u|^{2}\delta_{\alpha\beta}\psi-\frac{4}{(n-2)^{2}}u^{-\frac{2n}{n-2}}(u_{n}\psi_{n}+u_{\gamma}\psi_{\gamma})\delta_{\alpha\beta}\bigg\}\\
 & +\sum_{s=1}^{k-1}C_{1}(n,k,s)(2k-2s-1)\sigma_{s}\left(A_{g_{u}}^{\mathbb{\mathrm{T}}}\right)h_{g_{u}}^{2k-2s-2}\bigg\{\frac{2}{n-2}\frac{n}{n-2}u^{-\frac{2n-2}{n-2}}u_{n}\psi\\
 & -\frac{2}{n-2}u^{-\frac{n}{n-2}}\psi_{n}\bigg\}\\
=: & -b_{n}^{*}\psi_{n}-a_{\alpha\beta}^{*}\psi_{\alpha\beta}+\sum_{\alpha=1}^{n-1}b_{\alpha}^{*}\psi_{\alpha}+c^{*}\psi,
\end{align*}
where $b_{\alpha}^{*}$ and $c^{*}$ depend on $u_{1},u_{0}$.

Collecting all the coefficients of $\psi_{n}$ in $\ensuremath{\frac{\partial}{\partial t}}(\ensuremath{\mathcal{B}_{k}^{g_{u}}})$,
the definition of $b_{n}^{*}$ can be deduced as follows. 
\begin{align*}
 & C_{1}(n,k,0)(2k-1)h_{g_{u}}^{2k-2}\left(-\frac{2}{n-2}u^{-\frac{n}{n-2}}\right)\\
 & -\frac{2}{n-2}u^{-\frac{n}{n-2}}\sum_{s=1}^{k-1}C_{1}(n,k,s)(2k-2s-1)\sigma_{s}\left(A_{g_{u}}^{\mathbb{\mathrm{T}}}\right)h_{g_{u}}^{2k-2s-2}\\
 & -\frac{4}{(n-2)^{2}}u^{-\frac{2n}{n-2}}u_{n}\sum_{s=1}^{k-1}C_{1}(n,k,s)\frac{\partial}{\partial A_{u,\alpha\beta}^{\mathbb{\mathrm{T}}}}\sigma_{s}\left(A_{g_{u}}^{\mathbb{\mathrm{T}}}\right)\delta_{\alpha\beta}h_{g_{u}}^{2k-2s-1}\\
= & C_{1}(n,k,0)(2k-1)h_{g_{u}}^{2k-2}\left(-\frac{2}{n-2}u^{-\frac{n}{n-2}}\right)\\
 & -\frac{2}{n-2}u^{-\frac{n}{n-2}}\sum_{s=1}^{k-1}C_{1}(n,k,s)(2k-2s-1)\sigma_{s}\left(A_{g_{u}}^{\mathbb{\mathrm{T}}}\right)h_{g_{u}}^{2k-2s-2}\\
 & +\frac{2}{n-2}u^{-\frac{n}{n-2}}\sum_{s=1}^{k-1}C_{1}(n,k,s-1)(2k-2s+1)\sigma_{s-1}\left(A_{g_{u}}^{\mathbb{\mathrm{T}}}\right)h_{g_{u}}^{2k-2s}\\
= & -\frac{2}{n-2}u^{-\frac{n}{n-2}}\sigma_{k-1}\left(A_{g_{u}}^{\mathbb{\mathrm{T}}}\right)=-b_{n}^{*},
\end{align*}
where the first equality holds due to  $-\frac{4}{(n-2)^{2}}u^{-\frac{2n}{n-2}}u_{n}=\frac{2}{n-2}u^{-\frac{n}{n-2}}h_{g_{u}}$
and 
\[
\frac{\partial}{\partial A_{u,\alpha\beta}^{\mathbb{\mathrm{T}}}}\sigma_{s}\left(A_{g_{u}}^{\mathbb{\mathrm{T}}}\right)\delta_{\alpha\beta}=(n-s)\sigma_{s-1}\left(A_{g_{u}}^{\mathbb{\mathrm{T}}}\right).
\]
Now we finish the proof of the {\bf Claim}. 
Since ${g_{u}}\in\Gamma_{k}^{+}$ by the convexity of $\Gamma_k^+$ cone, we know that $\sigma_{s}(A_{g_{u}}^{\mathbb{\mathrm{T}}})>0$
for all $1\le s\le k-1$ and $\bigg\{\frac{\partial\sigma_{s}(A_{g_{u}}^{\mathbb{\mathrm{T}}})}{\partial A_{g_u,\alpha\beta}^{\mathbb{\mathrm{T}}}}\bigg\}_{(n-1)\times (n-1)}$
is positive-definite. Thus, the sign of $\mathcal{B}_{k}^{g_{u}}$
is consistent with $h_{g_{u}}$, keeping the ellipticity of $(\sigma_{k}^{g_{u}},\mathcal{B}_{k}^{g_{u}})$ on the boundary. Moreover,  $\{a_{\alpha\beta}\}$ and $b_n$ are positive and the ellipticity are strict.
\end{proof}

We remark here that the above proof of Lemma \ref{elliptic} also indicates that $\mathcal{B}_{k}^{g_u}$ is monotone with respect to $h_{g_u}$ if $g_u\in \Gamma_k^+$, and by formal computations,  $\frac{\partial \mathcal{B}_{k}^{g_u}}{\partial h_{g_u}}=\sigma_{k-1}(A_{g_u}^{\mathrm{T}})$. To see this,  we first rewrite  \eqref{tangential expression} as  $$A_{g_u}^{\mathrm{T}}=M-\frac 12 h_{g_u}^2\mathbb I_{(n-1)\times (n-1)},$$
where $M$ is a matrix depending only on the information of the tangential derivatives of $u$ on $\partial \mathbb{R}^n_+$ or the metric $g_u$ limiting on boundary.

 Now we find that 
 \begin{align}\label{monotone}
 \frac{\partial \mathcal{B}_{k}^{g_u}}{\partial h_{g_u}}
 =&\sum_{s=0}^{k-1}\frac{(n-1-s)!}{(n-k)!(2k-2s-1)!!}(2k-2s-1)\sigma_{s}\left(A_{g_u}^{\mathbb{\mathrm{T}}}\right)h_{g_u}^{2k-2s-2}\nonumber\\
 &+\sum_{s=0}^{k-1}\frac{(n-1-s)!}{(n-k)!(2k-2s-1)!!}\frac{\partial \sigma_{s}\left(A_{g_u}^{\mathbb{\mathrm{T}}}\right)}{\partial {A_{g_u}^{\mathbb{\mathrm{T}}}}_{,\alpha \beta}}(-h_{g_u}\delta_{\alpha\beta})h_{g_u}^{2k-2s-1}\nonumber\\
 =&\sigma_{k-1}(A_{g_u}^{\mathrm{T}}).
\end{align}

Therefore, if $g_{u}\in\Gamma_{k}^{+}$, then $\mathcal{B}_{k}^{g_{u}}$ is monotone with respect to $h_{g_u}$, and given a prescribed  $\mathcal{B}_{k}^{g_{u}}$ on boundary, the mean curvature $h_{g_u}$ can be uniquely determined once the tangential derivatives of function $u$ on boundary are described.

 In the left paper, we only focus on $g_{u}\in\Gamma_{k}^{+}$, and due to the positivity of $\mathcal{B}_{k}^{g_{u}}$, the corresponding $\{\tilde{a}_{ij}\}$, $\{a_{\alpha\beta}\}$  and $b_n$  are all positive,  ensuring that the strong maximum principle and the Hopf lemma works in the method of moving spheres. The formula \eqref{eq:linearization operator uner g-u^=00003D00007B4/(n-2)=00003D00007D} will be used in Section 3 and Section 4.

\section{Liouville Theorem under the Stronger condition}

In this section, we would like to prove a Liouville theorem in a stronger
assumption (\ref{eq:strong assmption}), which also naturally appears
in half sphere. Some lemmas in this section still work without (\ref{eq:strong assmption}).

In the following, we just consider positive constant $\mathcal{B}_{k}^{g_{u}}$
curvature. We use ${B}_{R}(x)$ to denote the ball in $\mathbb{R}^n$ of radius $R$ and centered at $x$, and write
$B_R = B_R(0)$ and ${B}_{R}^{+}=B_R\cap \mathbb{R}^n_+$. 
\begin{thm}
\label{thm:half space liouville-1}Let $g_{u}=u^{\frac{4}{n-2}}|dx|^{2}$
in $\mathbb{R}_{+}^{n}$ satisfy 
\begin{equation}
\begin{cases}
\sigma_{k}(A_{g_{u}})=2^{k}\binom{n}{k}, \quad g_{u}\in\Gamma_{k}^{+}, \qquad & \mathrm{in}\,\,\mathbb{R}_{+}^{n},\\
\mathcal{B}_{k}^{g_{u}}=c_{0}>0 & \text{\ensuremath{\mathrm{on}}}\,\,\partial\mathbb{R}_{+}^{n}.
\end{cases}\label{eq:main equation-1}
\end{equation}
Assume that 
\begin{equation}
u_{0,1}(x):=|x|^{2-n}u\bigg(\frac{x}{|x|^{2}}\bigg)\text{ can be extended to a positive function in }C^{2}\big(\overline{{B}_{1}^{+}}\big).\label{eq:strong assmption}
\end{equation}
Then, $u$ is the form of (\ref{eq:standard form}). 
\end{thm}

Theorem \ref{thm:half space liouville-1} will be used in the proof of Theorem \ref{thm:half space liouville} in Section 4.  The strong assumption \eqref{eq:strong assmption} describes the behavior of function $u$ near infinity, from which,  the functions such as ${u}_{x,1}(z)=\frac{1}{|z-x|^{n-2}}u(x+\frac{z-x}{|z-x|^2})$  can  also be extended as a positive function in  $C^2(\overline{{B}_{1}^{+}(x)})$. Since the $\mathcal{B}_{k}^{g}$ boundary condition is
conformally invariant, up to stereographic projection,  we can obtain Corollary \ref{corollary}.  This section is also to give the main ideas of the proof and may be easier to see how the $\mathcal{B}_k$ boundary curvature equations fit the method of moving spheres.  Readers can also see  Li-Li \cite{Li-Li2} for the classical argument for Neumann boundary condition.

For readers' convenience,
we write the following argument for $\mathcal{B}_{k}^{g}$,  which indicates that $\mathcal{B}_{k}^{g}$ is conformally invariant  and  will be frequently used throughout the paper.

Let $z=\varphi_{x,\lambda}(y)=x+\frac{\lambda^{2}(y-x)}{|y-x|^{2}}$
and then $\varphi_{x,\lambda}^{*}(|\text{d}z|^{2})=\frac{\lambda^{4}}{|y-x|^{4}}|\text{d}y|^{2}.$
For $g_{u}=u^{\frac{4}{n-2}}|\text{d}z|^{2}$, 
\begin{align*}
\varphi_{x,\lambda}^{*}(u^{\frac{4}{n-2}}|\text{d}z|^{2}) & =[u\circ\varphi_{x,\lambda}(y)]^{\frac{4}{n-2}}\frac{\lambda^{4}}{|y-x|^{4}}|\text{d}y|^{2}\\
 & =\bigg[\frac{\lambda^{n-2}u\circ\varphi_{x,\lambda}}{|y-x|^{n-2}}\bigg]^{\frac{4}{n-2}}|\text{d}y|^{2}.
\end{align*}
Now we denote $u_{x,\lambda}(y)=\frac{\lambda^{n-2}}{|y-x|^{n-2}}u\left(x+\frac{\lambda^{2}(y-x)}{|y-x|^{2}}\right)$
and from above, we obtain 
\begin{align*}
\varphi_{x,\lambda}^{*}(\mathcal{B}_{k}^{g_{u}}) & =\mathcal{B}_{k}^{\varphi_{x,\lambda}^{*}(g_{u})}=\mathcal{B}_{k}^{\varphi_{x,\lambda}^{*}(u^{\frac{4}{n-2}}|\text{d}z|^{2})}\\
 & =\mathcal{B}_{k}^{u_{x,\lambda}^{\frac{4}{n-2}}|\text{d}y|^{2}}.
\end{align*}
Thus, when $\mathcal{B}_{k}^{g_{u}}=c_{0}$, it yields that $\mathcal{B}_{k}^{u_{x,\lambda}^{\frac{4}{n-2}}|\text{d}y|^{2}}=c_{0}.$

For simplicity $A^{u}:=g_{u}^{-1}A_{g_{u}}$, $\mathcal{B}_{k}^{u}:=\mathcal{B}_{k}^{g_{u}}$
and $\partial'B_{r_{0}}^{+}:=\partial B_{r_{0}}^{+}\backslash\{x|x_{n}=0\}.$
Equation (\ref{eq:main equation-1}) is the same as (\ref{eq:main equation}). 
\begin{lem}
\label{lem:lower bound}Assume that $u$ is a solution to (\ref{eq:main equation-1}).
\textup{For any fixed $r_{0}>0$, when $|y|\ge r_{0}>0$ and }$y_{n}\ge0,$
we have \textup{ 
\[
u(y)\ge(\min_{\partial'B_{r_{0}}^{+}}u)r_{0}^{n-2}|y|^{2-n},
\]
}and

\[
\underset{x\in\mathbb{R}_{+}^{n},|x|\rightarrow+\infty}{\liminf}u|x|^{n-2}>0.
\]
\end{lem}

\begin{proof}
From the assumption, we know that $-\Delta u\ge0$ in $\overline{\mathbb{R}_{+}^{n}}\backslash B_{r_{0}}^{+},$
and $-\frac{\partial u}{\partial x_{n}}>0$ on $\partial\mathbb{R}_{+}^{n}\backslash\overline{B_{r_{0}}^{+}}$.
With the fact $\Delta|y|^{2-n}=0$ in $\mathbb{R}_{+}^{n}$ and $\frac{\partial|y|^{2-n}}{\partial y_{n}}=0$
on $\partial\mathbb{R}_{+}^{n}\backslash\{0\},$ by the maximum principle,
we have 
\[
u(y)\ge(\min_{\partial'B_{r_{0}}^{+}}u)r_{0}^{n-2}|y|^{2-n}\,\,\text{for }|y|\ge r_{0}\,\,\text{and}\,\,y_{n}\ge0.
\]
\end{proof}
For $x\in\partial\mathbb{R}_{+}^{n},\lambda>0$, let $u_{x,\lambda}$
denote the reflection of $u$ with respect to $B_{\lambda}(x)$, i.e.,
\[
u_{x,\lambda}(y):=\left(\frac{\lambda}{|y-x|}\right)^{n-2}u\left(x+\frac{\lambda^{2}(y-x)}{|y-x|^{2}}\right).
\]

By $g_{u}\in\Gamma_{k}^{+}$ and the conformal invariance of $A_{g}\in\Gamma_{k}^{+}$,
we know that $g_{u_{x,\lambda}}\in\Gamma_{k}^{+}\,\text{ on }\overline{B_{1}^{+}}$
and hence $\sigma_{s}(A_{g_{u_{x,\lambda}}}^{\mathbb{\mathrm{T}}})>0$
for $1\le s\le k-1$ and $\bigg\{\frac{\partial\sigma_{s}(A_{g_{u_{x,\lambda}}}^{\mathbb{\mathrm{T}}})}{\partial A_{\alpha\beta}^{\mathbb{\mathrm{T}}}}\bigg\}_{(n-1)\times(n-1)}$
is positive-definite. 

For simplicity, we denote $u_{\lambda}(y)=u_{0,\lambda}(y).$

$\vspace{2pt}$

For $x\in\partial\mathbb{R}_{+}^{n}$, define 
\[
\bar{\lambda}(x):=\sup\left\{ \mu>0\mid u_{x,\lambda}\leq u\text{ in }\overline{\mathbb{R}_{+}^{n}}\backslash B_{\lambda}(x),\quad\forall\,0<\lambda<\mu\right\} .
\]

$\vspace{2pt}$ 

Now we want to show that $\bar{\lambda}(x)>0.$
\begin{lem}
\label{lem:start}Assume that $u$ is a solution to (\ref{eq:main equation-1}).
Then, for any $x\in\partial\mathbb{R}_{+}^{n}$, there exists $\lambda_{0}(x)>0$
such that 
\[
u_{x,\lambda}\leq u\quad\text{\ensuremath{\mathrm{on}} }\,\,\overline{\mathbb{R}_{+}^{n}}\backslash B_{\lambda}(x),\quad\forall\,0<\lambda<\lambda_{0}(x).
\]
\end{lem}

\begin{proof}
Without loss of generality, we assume that $x=0.$ As $u\in C^{1}$,
there exists a positive constant $r_{0}$ such that for $0<r<r_{0}$
\begin{align}
\frac{d}{dr}\left(r^{(n-2)/2}u(r,\theta)\right) & =\frac{n-2}{2}r^{\frac{n-4}{2}}u+r^{\frac{n-2}{2}}u_{r}\nonumber \\
 & =r^{\frac{n-4}{2}}\big(\frac{n-2}{2}u+ru_{r}\big)>0.\label{eq:increasing near small r}
\end{align}
For any $\lambda$ satisfying $\lambda<|y|<r_{0},$ 
\[
\bigg|\frac{\lambda^{2}y}{|y|^{2}}\bigg|=\frac{\lambda^{2}}{|y|}<|y|<r_{0},
\]
and then, by (\ref{eq:increasing near small r}) 
\[
|y|^{(n-2)/2}u(|y|,\theta)>\bigg|\frac{\lambda^{2}y}{|y|^{2}}\bigg|^{(n-2)/2}u\left(\bigg|\frac{\lambda^{2}y}{|y|^{2}}\bigg|,\theta\right).
\]
Therefore, for $0<\lambda<|y|<r_{0},$

\begin{equation}
u(y)>u_{\lambda}(y)=\left(\frac{\lambda}{|y|}\right)^{n-2}u\left(\frac{\lambda^{2}y}{|y|^{2}}\right).\label{eq:annular domain}
\end{equation}
Taking $\lambda_{0}=\bigg(\frac{\min_{\partial'B_{r_{0}}^{+}}u}{\max_{\overline{B_{r_{0}}^{+}}}u}\bigg)^{\frac{1}{n-2}}r_{0}$,
for $0<\lambda<\lambda_{0}$ and $|y|\ge r_{0}$, 
\begin{align*}
u_{\lambda}(y) & \le\left(\frac{\lambda_{0}}{|y|}\right)^{n-2}\max_{\overline{B_{\lambda^{2}/r_{0}}^{+}}}u\le\left(\frac{\lambda_{0}}{|y|}\right)^{n-2}\max_{\overline{B_{r_{0}}^{+}}}u\le\frac{r_{0}^{n-2}\min_{\partial'B_{r_{0}}^{+}}u}{|y|^{n-2}}\le u(y),
\end{align*}
where the first inequality is due to the definition of $u_{\lambda}$
and the last inequality holds due to Lemma \ref{lem:lower bound}.

Combining with (\ref{eq:annular domain}), we know that for $0<\lambda<\lambda_{0}$
and $y\in\overline{\mathbb{R}_{+}^{n}}\backslash B_{\lambda}(x),$
\[
u_{\lambda}(y)\le u(y).
\]
\end{proof}
From Lemma \ref{lem:start}, $\bar{\lambda}(x)>0$. Now we
want to prove that $\bar{\lambda}(x)<\infty.$
\begin{lem}
Given the same assumptions as Theorem \ref{thm:half space liouville-1},
we have $\bar{\lambda}(x)<\infty$ for any $x\in\partial\mathbb{R}_{+}^{n}$. 
\end{lem}

\begin{proof}
As $u_{0,1}$ can be extended to a positive continuous function near
zero, we have 
\[
|x|^{2-n}u\left(\frac{x}{|x|^{2}}\right)\rightarrow\alpha_{0}\quad\text{as}\quad x\rightarrow0.
\]
And then there exist positive constants $\alpha_{1}$ and $r_{1}$
such that 
\[
u(y)\le\frac{\alpha_{1}}{|y|^{n-2}}\quad\text{for}\,\,|y|\ge r_{1}.
\]
By the definition of $\bar{\lambda}$, for any $0<\lambda<\bar{\lambda}$,  we have 
\[
\lambda^{n-2}u(x)=\lim_{y\rightarrow\infty}\lambda^{n-2}u\left(x+\frac{\lambda^{2}(y-x)}{|y-x|^{2}}\right)\le\lim_{y\rightarrow\infty}|y-x|^{n-2}u(y)\le\alpha_{1}.
\]
Thus, 
\[
\bar{\lambda}(x)\le C.
\]
\end{proof}
\begin{lem}
\label{lem:boundary kelvin equal}Given the same assumptions as Theorem
\ref{thm:half space liouville-1}, we have that, for all $x\in\partial\mathbb{R}_{+}^{n}$,
\[
u_{x,\bar{\lambda}(x)}\equiv u\quad\text{ \ensuremath{\mathrm{in}} }\mathbb{R}_{+}^{n}\backslash\{x\}.
\]
\end{lem}

\begin{proof}
We argue by contradiction. Note that the assumption \eqref{eq:strong assmption} actually describes the behavior of $u$ near infinity and the functions such as ${u}_{x,1}(z)=\frac{1}{|z-x|^{n-2}}u(x+\frac{z-x}{|z-x|^2})$  can  also be extended as a positive function in  $C^2(\overline{{B}_{1}^{+}(x)})$. Now,  without loss of generality,
we take $x=0$ and $u_{0,\bar{\lambda}}\not\equiv u$ in $\mathbb{R}_{+}^{n}\backslash\{0\}$,
where $\bar{\lambda}:=\bar{\lambda}(0)$ (by abuse of notation). We know
that $u_{\bar{\lambda}}\le u$ on $\overline{\mathbb{R}_{+}^{n}}\backslash B_{\bar{\lambda}}$
by the definition of $\bar{\lambda}.$

We know $\sigma_{k}(A^{u_{\bar{\lambda}}})=2^{k}\binom{n}{k}$ on $\overline{\mathbb{R}_{+}^{n}}\backslash\overline{B_{\bar{\lambda}}^{+}}$
and $\mathcal{B}_{k}^{u_{\bar{\lambda}}}=c_{0}\,\,\text{ on }\partial\mathbb{R}_{+}^{n}\backslash B_{\bar{\lambda}}.$
Letting $w=tu+(1-t)u_{\bar{\lambda}}$, we know that $u-u_{\bar{\lambda}}$
satisfies

\[
\begin{cases}
0=F(A^{u})-F(A^{u_{\bar{\lambda}}})=\mathbb{L}(u-u_{\bar{\lambda}}) & \text{in}\,\,\overline{\mathbb{R}_{+}^{n}}\backslash B_{\bar{\lambda}},\\
0=\mathcal{B}_{k}^{u}-\mathcal{B}_{k}^{u_{\bar{\lambda}}}=L(u-u_{\bar{\lambda}}) & \text{on}\,\,\partial\mathbb{R}_{+}^{n}\backslash B_{\bar{\lambda}},
\end{cases}
\]
where 
\[
\mathbb{L}\varphi:=-A_{ij}\partial_{ij}\varphi+B_{i}\partial_{i}\varphi+C(x)\varphi,
\]

\[
A_{ij}:=\frac{2}{n-2}\int_{0}^{1}w^{-\frac{n+2}{n-2}}\frac{\partial\sigma_{k}}{\partial A_{ij}^{w}}\left(g_{w}^{-1}A^{w}\right)dt,
\]
and $L\varphi$ is defined in (\ref{eq:linearization operator uner g-u^=00003D00007B4/(n-2)=00003D00007D})
with $u_{0}=u_{\bar{\lambda}}$ and $u_{1}=u$ in (\ref{eq:linearization operator uner g-u^=00003D00007B4/(n-2)=00003D00007D}).
We first provide two claims for later use. 

\emph{Claim 1:} 
\begin{equation}
u-u_{\bar{\lambda}}>0\text{ on }\overline{\mathbb{R}_{+}^{n}}\backslash\overline{B_{\bar{\lambda}}}.\label{eq:Strict positive for moving forward a little}
\end{equation}

\emph{Proof of Claim 1}: We argue by the contradiction and assume
$(u-u_{\bar{\lambda}})(\overline{x})=0$ for some $\overline{x}\in\overline{\mathbb{R}_{+}^{n}}\backslash\overline{B_{\bar{\lambda}}}$.
It holds that 
\begin{equation}
\begin{cases}
\mathbb{L}(u-u_{\bar{\lambda}})=0 & \text{in\,}\,\overline{\mathbb{R}_{+}^{n}}\backslash B_{\bar{\lambda}},\\
L(u-u_{\bar{\lambda}})=0 & \text{on}\,\,\partial\mathbb{R}_{+}^{n}\backslash B_{\bar{\lambda}}.
\end{cases}\label{eq:difference equation}
\end{equation}

If $\overline{x}\in\mathbb{R}_{+}^{n}\backslash\overline{B_{\bar{\lambda}}}$,
then by the strong maximum principle, $u-u_{\bar{\lambda}}=0$ near
$\overline{x},$ and it follows that $u=u_{\overline{\lambda}}$ on
$\overline{\mathbb{R}_{+}^{n}}$, leading to a contradiction.

If $\overline{x}$ is on the boundary, then from Hopf Lemma that $\partial_{n}(u-u_{\bar{\lambda}})>0$
at $\overline{x}$. Moreover, at $\overline{x}$, 
\[
\partial_{\alpha\beta}(u-u_{\bar{\lambda}})\ge0,\,\,\partial_{\alpha}(u-u_{\bar{\lambda}})=0,\,\,u-u_{\bar{\lambda}}=0,
\]
and then 
\begin{align*}
0 & =L(u-u_{\bar{\lambda}})\\
 & =-a_{\alpha\beta}\partial_{\alpha\beta}(u-u_{\bar{\lambda}})+b_{\alpha}\partial_{\alpha}(u-u_{\bar{\lambda}})-b_{n}\partial_{n}(u-u_{\bar{\lambda}})+c(u-u_{\bar{\lambda}})\\
 & \le-b_{n}\partial_{n}(u-u_{\bar{\lambda}})<0.
\end{align*}
 A contradiction! We have proved Claim 1.

\emph{Claim 2: } 
\begin{equation}
\lim_{y\in\overline{\mathbb{R}_{+}^{n}},|y|\rightarrow\infty}|y|^{n-2}\left(u(y)-u_{\bar{\lambda}}(y)\right)>0.\label{eq:asymptotic infinity behavior}
\end{equation}
\emph{ Proof of Claim 2}: Let $x=y/|y|^{2},$ we have 
\[
|y|^{n-2}u(y)=u_{0,1}(x),\quad|y|^{n-2}u_{\bar{\lambda}}(y)=\bar{\lambda}^{n-2}u\left(\frac{\bar{\lambda}^{2}y}{|y|^{2}}\right)=\bar{\lambda}^{n-2}u\left(\bar{\lambda}^{2}x\right)=:v(x).
\]

By (\ref{eq:Strict positive for moving forward a little}), $u_{0,1}-v>0$
in $B_{1/\bar{\lambda}}^{+}.$ We prove that $(u_{0,1}-v)(0)>0$ by
a contradiction argument and then the Claim 2 will be proved. Otherwise,
as $u_{0,1}(x)$ and $v(x)$ satisfy the equation (\ref{eq:main equation}),
by (\ref{eq:strong assmption}) and Hopf Lemma, we know \footnote{Here we used the assumption that $u_{0,1}(x)$ is $C^{2}$ near $0$.} 
\[
\min_{\overline{B_{1/\bar{\lambda}}^{+}}}(u_{0,1}-v)=(u_{0,1}-v)(0)=0,
\]
and \[\partial_{\alpha}(u_{0,1}-v)(0)=0,\,\,\partial_{\alpha\beta}(u_{0,1}-v)(0)\ge0, \partial_{n}(u_{0,1}-v)(0)>0.\]
Thus, a contradiction follows from 
\begin{align*}
0 & =L(u_{0,1}-v)(0)\\
 & =-a_{\alpha\beta}\partial_{\alpha\beta}(u_{0,1}-v)+b_{\alpha}\partial_{\alpha}(u_{0,1}-v)-b_{n}\partial_{n}(u_{0,1}-v)+c(u_{0,1}-v)\\
 & \le-b_{n}\partial_{n}(u_{0,1}-v)<0.
\end{align*}
This completes the proof of Claim 2.

$\vspace{2pt}$

Since $u-u_{\bar{\lambda}}=0$ on $\partial B_{\bar{\lambda}}\cap\mathbb{R}_{+}^{n}$
from the definition of $u_{\bar{\lambda}}$, by the Hopf Lemma and
Claim 1, 
\begin{equation}
(u-u_{\bar{\lambda}})_{\nu}>0,\label{eq:move further-condition1}
\end{equation}
where $\nu$ denotes the outward unit  normal to $\partial B_{\bar{\lambda}}\cap\mathbb{R}_{+}^{n}.$

$\vspace{1pt}$

For any $\overline{x}\in\partial B_{\bar{\lambda}}\cap\partial\mathbb{R}_{+}^{n}$,
taking $\Omega=\big(\mathbb{R}_{+}^{n}\backslash B_{\bar{\lambda}}\big)\cap B_{1}(\overline{x})$
and $\sigma=x_{n}$.  In  $\Omega$, there exists a positive constant $A$ such that $C(x)\le A$ and $c(x)\le A$. And then
\begin{align*}
\mathscr{P}(u-u_{\bar{\lambda}})&=\mathbb{L}(u-u_{\bar{\lambda}})-C(x)(u-u_{\bar{\lambda}})\\
&=-C(x)(u-u_{\bar{\lambda}})\ge -A(u-u_{\bar{\lambda}}),\quad\text{in}\,\,\Omega
\end{align*}
and 
\begin{align*}
\mathscr{L}(u-u_{\bar{\lambda}})&=L(u-u_{\bar{\lambda}})-c(x)(u-u_{\bar{\lambda}})\\
&=-c(x)(u-u_{\bar{\lambda}})\ge-A(u-u_{\bar{\lambda}}),\quad\text{on}\,\,\{\sigma=0\}\cap\overline{\Omega}.
\end{align*}
By Lemma \ref{lem:corner hopf } in Appendix, we know that $\frac{\partial}{\partial\nu}(u-u_{\bar{\lambda}})\big|_{x=\overline{x}}>0$,
where $\nu$ denotes the unit outer normal to $\partial B_{\bar{\lambda}}$
on $\partial\mathbb{R}_{+}^{n}.$ Thus, we have 
\begin{equation}
\frac{\partial}{\partial\nu}(u-u_{\bar{\lambda}})>0\quad\text{on}\quad\partial B_{\bar{\lambda}}\cap\partial\mathbb{R}_{+}^{n},\label{eq:move further-condition1-1}
\end{equation}
where $\nu$ denotes the unit outer normal to $\partial B_{\bar{\lambda}}$
on $\partial\mathbb{R}_{+}^{n}.$

Therefore, from (\ref{eq:move further-condition1})(\ref{eq:move further-condition1-1}),
there exists a positive constant $b>0$ such that 
\[
\frac{\partial}{\partial\nu}(u-u_{\bar{\lambda}})>b>0\quad\text{on}\quad\partial B_{\bar{\lambda}}\cap\overline{\mathbb{R}_{+}^{n}}.
\]
By the continuity of $\nabla u$, there exists a $R>\bar{\lambda}$
such that for $\bar{\lambda}\le\lambda\le R$, $\lambda\le|x|\le R$,
we have 
\[
\frac{\partial}{\partial\nu}(u-u_{\lambda})\bigg|_{x}>\frac{b}{2},
\]
where $\nu$ is the unit outer normal to $\partial B_{\lambda}.$
For $\bar{\lambda}\le\lambda\le R$ and $\lambda<|y|\le R$, as $u-u_{\lambda}=0$
on $\partial B_{\lambda}$, from above, 
\begin{equation}
u(y)-u_{\lambda}(y)>0.\label{eq:move further more in small domain}
\end{equation}

By Claim 2, we know that there exists a positive constant $C_{0}>0$
such that for $|y|\ge R$ and $y_{n}\ge0$, 
\[
u(y)-u_{\bar{\lambda}}(y)\ge\frac{C_{0}}{|y|^{n-2}}.
\]
Furthermore, for $|y|\ge R,$ 
\begin{align}
u(y)-u_{\lambda}(y) & \ge\frac{C_{0}}{|y|^{n-2}}-(u_{\lambda}-u_{\bar{\lambda}})\nonumber \\
 & =\frac{C_{0}}{|y|^{n-2}}-\frac{1}{|y|^{n-2}}\bigg(\lambda^{n-2}u\left(\frac{\lambda^{2}y}{|y|^{2}}\right)-\bar{\lambda}^{n-2}u\left(\frac{\bar{\lambda}^{2}y}{|y|^{2}}\right)\bigg)\nonumber \\
 & \ge\frac{1}{|y|^{n-2}}\bigg[C_{0}-\lambda^{n-2}u\left(\frac{\lambda^{2}y}{|y|^{2}}\right)+\bar{\lambda}^{n-2}u\left(\frac{\bar{\lambda}^{2}y}{|y|^{2}}\right)\bigg]\nonumber \\
 & \ge\frac{\frac{C_{0}}{2}}{|y|^{n-2}},\label{eq:move further more in large domain}
\end{align}
where the last inequality holds as we takes $\lambda$ close to $\bar{\lambda}$.
By (\ref{eq:move further more in small domain}) and (\ref{eq:move further more in large domain}),
there exists some small $\varepsilon>0$ such that for $\bar{\lambda}\le\lambda\le\bar{\lambda}+\varepsilon$
and $|y|\ge\lambda,$ we have $u(y)-u_{\lambda}(y)>0$. It will yield
a contradiction to the largest number of $\bar{\lambda}.$ Therefore,
$u_{\bar{\lambda}}=u$ in $\overline{\mathbb{R}_{+}^{n}}\backslash B_{\bar{\lambda}}$
and thus $u_{\bar{\lambda}}=u$ in $\overline{\mathbb{R}_{+}^{n}}\backslash\{0\}$. 
\end{proof}

\bigskip

\begin{proof}[\textbf{Proof of Theorem \ref{thm:half space liouville-1}}]
Now by Lemma \ref{lem:Cal Lemma-entire space} in Appendix and Lemma
\ref{lem:boundary kelvin equal}, we know that on $\partial\mathbb{R}_{+}^{n}$,
there exist $a>0$, $d>0$, and $x_{0}'\in\partial\mathbb{R}_{+}^{n}$
such that 
\begin{equation}
u(x',0)=\left(\frac{a}{d^{2}+|x'-x_{0}'|^{2}}\right)^{\frac{n-2}{2}}.\label{eq:function on boundary}
\end{equation}

Let $p=\left(x_{0}^{\prime},-d\right)$ and $q=\left(x_{0}^{\prime},d\right)$,
and define 
\[
y=(y',y_{n})=p+\frac{4d^{2}(z-p)}{|z-p|^{2}}:B_{2d}(q)\longrightarrow\mathbb{R}_{+}^{n}
\]
and 
\[
v(z):=\left(\frac{2d}{|z-p|}\right)^{n-2}u\left(p+\frac{4d^{2}(z-p)}{|z-p|^{2}}\right).
\]
Now $v$ satisfies 
\begin{equation}
\begin{cases}
\sigma_{k}(A_{g_{v}})=2^{k}\binom{n}{k} \qquad & \text{in}\,\,B_{2d}(q),\\
\mathcal{B}_{k}^{g_{v}}=c_{0} & \text{on}\,\,\partial B_{2d}(q).
\end{cases}\label{eq:radial solution condition}
\end{equation}

We know that $\varphi(\partial B_{2d}(q))=\partial\mathbb{R}_{+}^{n}$ where
$\varphi(z)=p+\frac{4d^{2}(z-p)}{|z-p|^{2}}.$ Moreover, when $z\in\partial B_{2d}(q)$,
we have $y=(y',0)$ and 
\begin{align}
v(z) & =\left(\frac{2d}{|z-p|}\right)^{n-2}u\left(p+\frac{4d^{2}(z-p)}{|z-p|^{2}}\right)\label{eq:dirichlet boundary}\\
 & =\left(\frac{2d}{|z-p|}\right)^{n-2}u(y',0)\nonumber \\
 & =\left(\frac{2d}{|z-p|}\right)^{n-2}\left(\frac{a}{d^{2}+|y'-x_{0}'|^{2}}\right)^{\frac{n-2}{2}}\nonumber \\
 & =a^{\frac{n-2}{2}}\left(\frac{2d}{|z-p|}\right)^{n-2}\frac{1}{|y-p|^{n-2}}\nonumber \\
 & =a^{\frac{n-2}{2}}(2d)^{2-n},\nonumber 
\end{align}
where we use (\ref{eq:function on boundary}) in the third equality
and the last equality holds due to $|y-p||z-p|=4d^{2}$ for $z\in\partial B_{2d}(q)$.
Notice that the expression of \eqref{eq:boundary equation} consists of the tangential derivatives and the first normal derivatives on boundary.  By  (\ref{eq:dirichlet boundary}), we have the tangential information of $v$ on $\partial B_{2d}(q)$ and then by \eqref{eq:boundary equation}, we can have the first normal derivative on boundary. 
Actually, from (\ref{eq:dirichlet boundary})\eqref{monotone} and $\mathcal{B}_{k}^{g_{v}}=c_{0}$
on $\partial B_{2d}(q)$, we obtain that $\left<\nabla v,\nu\right>$
is constant on $\partial B_{2d}(q)$, where $\nu$ is the unit outer
normal vector on $\partial B_{2d}(q)$. With constant Dirichlet and
Neumann boundary on $\partial B_{2d}(q)$, the solution $v$ to $\sigma_{k}(A_{g_{v}})=2^{k}\binom{n}{k}$
on $B_{2d}(q)$ is radial by the standard moving plane argument. Then,
by Theorem \ref{thm:Li-Li-JEMS Thm3} in Appendix, we conclude that
\[
v(z)=\bigg(\frac{\sqrt{b}}{1+b|z-q|^{2}}\bigg)^{\frac{n-2}{2}}
\]
satisfying (\ref{eq:radial solution condition}). Thus, we obtain the
form of (\ref{eq:standard form}) by transforming back to $u,$ which
means 
\[
u(y)=\bigg(\frac{2d}{|y-p|}\bigg)^{n-2}v\left(p+\frac{4d^{2}(y-p)}{|y-p|^{2}}\right).
\]
\end{proof}

\section{Proof of Theorem \ref{thm:half space liouville}}

In this section,   under the assumptions of Theorem \ref{thm:half space liouville}, we will utilize the beautiful lemma in \cite{Li-Li}
to  obtain (\ref{eq:strong assmption}) and thereby, we  complete
the proof of Theorem \ref{thm:half space liouville}.
\begin{lem}
\label{lem: how far can bar=00003D00007B=00003D00005Clambda=00003D00007D be?}Assume
as Theorem \ref{thm:half space liouville}. Denote $\alpha:=\lim\inf_{x\in\overline{\mathbb{R}_{+}^{n}},|x|\rightarrow+\infty}|x|^{n-2}u.$
We have $0<\alpha<+\infty$, and 
\[
\bar{\lambda}(x)^{n-2}u(x)=\alpha,\quad \forall~ x\in\partial\mathbb{R}_{+}^{n}.
\]
\end{lem}

\begin{proof}
By definition of $\bar{\lambda}(x)$ we have 
\[
u_{x,\lambda}\leq u\quad\text{ in }\overline{\mathbb{R}_{+}^{n}}\backslash B_{\lambda}(x),\quad\forall\,0<\lambda<\bar{\lambda}(x).
\]
For $\forall\,0<\lambda<\bar{\lambda}(x),$ 
\begin{align*}
\lambda^{n-2}u(x) & =\liminf_{y\rightarrow+\infty,y\in\overline{\mathbb{R}_{+}^{n}}}|y|^{n-2}\bigg(\frac{\lambda}{|y-x|}\bigg)^{n-2}u\left(x+\frac{\lambda^{2}(y-x)}{|y-x|^{2}}\right)\\
 & \le\liminf_{y\rightarrow+\infty,y\in\overline{\mathbb{R}_{+}^{n}}}|y|^{n-2}u(y)=\alpha.
\end{align*}
Then $\bar{\lambda}(x)^{n-2}u(x)\le\alpha$ for $\forall\,x\in\partial\mathbb{R}_{+}^{n}$.
Now we prove $\bar{\lambda}(x)^{n-2}u(x)=\alpha$ for $\forall\,x\in\partial\mathbb{R}_{+}^{n}$
by a contradiction argument.

We divide the following proof into two cases: $\alpha<+\infty$ and $\alpha=+\infty$, and will exclude the case  $\alpha=+\infty$. 

{\emph {Case 1}}:  $\alpha<+\infty$.

\emph{Step 1:} If there exists a $x\in\partial\mathbb{R}_{+}^{n}$
such that $\bar{\lambda}(x)^{n-2}u(x)<\alpha$, then 
\begin{align}
 & \liminf_{y\rightarrow+\infty,y\in\overline{\mathbb{R}_{+}^{n}}}|y|^{n-2}(u(y)-u_{x,\bar{\lambda}(x)}(y))\label{eq:asymptotic behavior}\\
 & \ge\alpha-\limsup_{y\rightarrow+\infty,y\in\overline{\mathbb{R}_{+}^{n}}}|y|^{n-2}\left(\frac{\bar{\lambda}(x)}{|y-x|}\right)^{n-2}u\left(x+\frac{\bar{\lambda}(x)^{2}(y-x)}{|y-x|^{2}}\right)\nonumber \\
 & =\alpha-\bar{\lambda}(x)^{n-2}u(x)>0.\nonumber 
\end{align}

\emph{Step 2:} We prove that 
\begin{equation}
u-u_{x,\bar{\lambda}(x)}>0\text{ in }\overline{\mathbb{R}_{+}^{n}}\backslash\overline{B_{\bar{\lambda}(x)}}(x).\label{eq:Strict positive for moving forward a little-1}
\end{equation}

We argue by the contradiction argument and without loss of generality,
assume $x=0$ and $(u-u_{\bar{\lambda}})(\overline{x})=0$ for some
$\overline{x}\in\overline{\mathbb{R}_{+}^{n}}\backslash\overline{B_{\bar{\lambda}}}$,
where $\bar{\lambda}:=\bar{\lambda}(0)$. We know $\sigma_{k}(A^{u_{\text{\ensuremath{\bar{\lambda}}}}})=2^{k}\binom{n}{k}$
on $\overline{\mathbb{R}_{+}^{n}}\backslash\overline{B_{\bar{\lambda}}}$
and $\mathcal{B}_{k}^{u_{\bar{\lambda}}}=c_{0}\,\text{ on }\partial\mathbb{R}_{+}^{n}\backslash B_{\bar{\lambda}}.$
It holds that
\[
\begin{cases}
0=F(A^{u})-F(A^{u_{\overline{\lambda}}})=\mathbb{L}(u-u_{\overline{\lambda}}), & \text{in}\quad\overline{\mathbb{R}_{+}^{n}}\backslash\overline{B_{\bar{\lambda}}},\\
0=\mathcal{B}_{k}^{u}-\mathcal{B}_{k}^{u_{\overline{\lambda}}}=L(u-u_{\overline{\lambda}}) & \text{on}\quad\partial\mathbb{R}_{+}^{n}\backslash B_{\bar{\lambda}}.
\end{cases}
\]
If $\overline{x}\in\mathbb{R}_{+}^{n}\backslash\overline{B_{\bar{\lambda}}}$,
then by the strong maximum principle, $u-u_{\bar{\lambda}}=0$ near
$\overline{x}$ and then $u=u_{\bar{\lambda}}$ on $\overline{\mathbb{R}_{+}^{n}}$,
which is contradicted to (\ref{eq:asymptotic behavior}). If $\overline{x}$
is on boundary, then by Hopf Lemma, $\partial_{n}(u-u_{\bar{\lambda}})\bigg|_{\overline{x}}>0$.

At $\overline{x}$, 
\[
\partial_{\alpha\beta}(u-u_{\bar{\lambda}})\ge0,\,\,\partial_{\alpha}(u-u_{\bar{\lambda}})=0,\,\,u-u_{\bar{\lambda}}=0
\]
and then 
\begin{align*}
0 & =L(u-u_{\bar{\lambda}})\bigg|_{x=\overline{x}}\\
 & =-a_{\alpha\beta}\partial_{\alpha\beta}(u-u_{\bar{\lambda}})+b_{\alpha}\partial_{\alpha}(u-u_{\bar{\lambda}})-b_{n}\partial_{n}(u-u_{\bar{\lambda}})+c(u-u_{\bar{\lambda}})\\
 & \le-b_{n}\partial_{n}(u-u_{\bar{\lambda}})<0,
\end{align*}
which yields a contradiction. We have proved (\ref{eq:Strict positive for moving forward a little-1}).

$\vspace{0.5pt}$
Noting that in the proof of Lemma \ref{lem:boundary kelvin equal}, once we obtain Claim 1 and  Claim 2, the remaining proof has no relationship with the assumption \eqref{eq:strong assmption}.
Now with \eqref{eq:asymptotic behavior} and (\ref{eq:Strict positive for moving forward a little-1}) established in Step 1 and Step 2,  the remaining argument is same as the
proof of Lemma \ref{lem:boundary kelvin equal} after Claim 2, and  it will
yield a contradiction to the largest number of $\bar{\lambda}.$ Thus
\[
\bar{\lambda}(x)^{n-2}u(x)=\alpha.
\]

{\emph{Case 2}:}$\alpha=+\infty$.

 By repeating the argument in Step 1 and Step 2 in {\emph{Case 1}}, it holds that 
\[
\bar{\lambda}(x)=+\infty,\,\forall x\in\partial\mathbb{R}_{+}^{n},
\]
from which 
\[
u_{x,\lambda}\le u\,\,\text{on}\,\,\mathbb{R}_{+}^{n}\backslash B_{\lambda}(x)\quad\text{for}\,\,0<\lambda<+\infty.
\]
By Lemma \ref{lem:single variable function} in Appendix, it holds
that $u(x',x_{n})=u(0,x_{n})$ for $x_{n}\ge0.$ By $Tr(A^u)^{\mathrm{T}}>0$
on $\partial\mathbb{R}_{+}^{n}$, we have the following contradiction:
\begin{align*}
0 & <-\frac{2}{n-2}u^{-(n+2)/(n-2)}\overline{\Delta}u+\frac{2n}{(n-2)^{2}}u^{-2n/(n-2)}|\overline{\nabla}u|^{2}-\frac{2(n-1)}{(n-2)^{2}}u^{-2n/(n-2)}|\nabla u|^{2}\\
 & =-\frac{2(n-1)}{(n-2)^{2}}u^{-2n/(n-2)}u_{n}^{2}\le0\:\,\text{on}\,\,\partial\mathbb{R}_{+}^{n},
\end{align*}
where $\overline{\Delta}u=\sum_{\beta=1}^{n-1}u_{\beta\beta}=0$ and
$|\overline{\nabla}u|^{2}=\sum_{\beta=1}^{n-1}u_{\beta}^{2}=0$. Therefore, {\emph{Case 2}} can not happen.

Above all, we have completed the proof of the lemma. 
\end{proof}

Until now, we just use $\lim\inf_{x\in\overline{\mathbb{R}_{+}^{n}},|x|\rightarrow+\infty}|x|^{n-2}u$ and have not used the assumption that $\lim_{x\rightarrow0}u_{0,1}$  exists in Theorem \ref{thm:half space liouville}. In the left paragraph, we will need this assumption, see equation (\ref{eq:limit at zero}) and \eqref{eq:limit at zero 2} below.

Let us recall the essential lemma in \cite{Li-Li}. 
\begin{lem}[Li-Li,Acta Math.2005]
\label{lem:lift lemma} For $n\ge2$, $u\in L_{\text{loc}}^{1}(\mathbb{B}_{1}^{n}\backslash\{0\})$
is the solution of $\Delta u\le0$ in $\mathbb{B}_{1}^{n}\backslash\{0\}$
in the distribution sense. Assume that there exists $a\in\mathbb{R}$
and $p\neq q\in\mathbb{R}^{n}$ such that 
\[
u(x)\ge\max\{a+p\cdot x-\delta(x),a+q\cdot x-\delta(x)\},\quad x\in\mathbb{B}_{1}^{n}\backslash\{0\},
\]
where $\delta(x)\ge0$ satisfies $\lim_{x\rightarrow0}\frac{\delta(x)}{|x|}=0.$
Then 
\[
\lim_{r\rightarrow0}\inf_{B_{r}}u>a.
\]
\end{lem}

For simplicity, we denote $u_{\psi}:=|J_{\psi}|^{\frac{n-2}{2n}}u\circ\psi$,
where $J_{\psi}$ denotes the Jacobian of $\psi.$ Let us complete
the proof of Theorem \ref{thm:half space liouville} as the following: 
\begin{proof}[Proof of Theorem \ref{thm:half space liouville}]
Let $\varphi^{(x)}(y)=x+\frac{\bar{\lambda}(x)^{2}(y-x)}{|y-x|^{2}}$
and $z=\frac{y}{|y|^{2}}=:\psi(y)$. Note that $u_{x,\bar{\lambda}(x)}(y)=u_{\varphi^{(x)}}(y)$
and 
\[
u_{\varphi^{(x)}\circ\psi}(y)=\frac{1}{|y|^{n-2}}\frac{\bar{\lambda}(x)^{n-2}}{|\psi(y)-x|^{n-2}}u\left(x+\frac{\bar{\lambda}^{2}(x)(\psi(y)-x)}{|\psi(y)-x|^{2}}\right).
\]
It follows that
\begin{align}
u_{\varphi^{(x)}\circ\psi}(0) & =\lim_{y\rightarrow0}u_{\varphi^{(x)}\circ\psi}(y)\label{eq:zero point equal}\\
 & =\bar{\lambda}(x)^{n-2}u(x)=\alpha,\nonumber 
\end{align}
where the last equality is due to Lemma \ref{lem: how far can bar=00003D00007B=00003D00005Clambda=00003D00007D be?}
and $\alpha=\lim\inf_{x\in\overline{\mathbb{R}_{+}^{n}},|x|\rightarrow+\infty}|x|^{n-2}u$ by  Lemma \ref{lem: how far can bar=00003D00007B=00003D00005Clambda=00003D00007D be?}.

Denote $w^{(x)}(y):=u_{\varphi^{(x)}\circ\psi}(y)$. By the conformal
invariance of equation (\ref{eq:main equation}), $w^{(x)}(y)$ satisfies
\[
\begin{cases}
\sigma_{k}(A^{w^{(x)}})=2^{k}\binom{n}{k} \qquad& \text{in}\,\,\mathbb{R}_{+}^{n},\\
\mathcal{B}_{k}^{w^{(x)}}=c_{0} & \text{on}\,\,\partial\mathbb{R}_{+}^{n}\backslash\{\frac{x}{|x|^{2}}\}.
\end{cases}
\]
Also, $w^{(x)}(0)=\bar{\lambda}(x)^{n-2}u(x)=\alpha$ for $\,x\in\partial\mathbb{R}_{+}^{n}$.
Meanwhile, $u_{\psi}(y)=\frac{1}{|y|^{n-2}}u(\frac{y}{|y|^{2}})$
satisfies 
\[
\begin{cases}
\sigma_{k}(A^{u_{\psi}})=2^{k}\binom{n}{k} \qquad & \text{in}\,\,\mathbb{R}_{+}^{n},\\
\mathcal{B}_{k}^{u_{\psi}}=c_{0} & \text{on}\,\,\partial\mathbb{R}_{+}^{n}\backslash\{0\}.
\end{cases}
\]
Since we assume that $\lim_{x\rightarrow0}u_{0,1}$ exists, with the
definition of $\alpha,$ we obtain 
\begin{equation}
\lim_{y\rightarrow0}u_{\psi}(y)=\alpha.\label{eq:limit at zero}
\end{equation}

From the definition of $\bar{\lambda}(x),$ 
\[
u(y)\ge u_{\varphi^{(x)}}(y)\quad\text{on}\,\,\overline{\mathbb{R}_{+}^{n}}\backslash B_{\bar{\lambda}(x)}(x)\quad\text{for}\,x\in\partial\mathbb{R}_{+}^{n}.
\]
Thus, $u_{\psi}(z)\ge u_{\varphi^{(x)}\circ\psi}(z)$, where $z\in \{z|\frac{z}{|z|^{2}}\in\mathbb{R}_{+}^{n}\backslash B_{\bar{\lambda}(x)}(x)\}$, 
from which, there exists $\delta(x)>0$ depending on $\bar{\lambda}(x)$
and $x$ such that 
\[
u_{\psi}(z)\ge w^{(x)}(z)\quad \,\,\text{in}\quad \overline{B_{\delta(x)}^{+}}\backslash\{0\}.
\]

As $w^{(x)}$ is $C^{2}$ near $0$ and $u_{\psi}$ is $C^{2}$ away
from $0$ in $\overline{\mathbb{R}_{+}^{n}}\backslash\{0\}$, we know
that 
\[
u_{\psi}(z)\ge w^{(x)}(z)\quad \,\,\text{in}\,\,B_{\delta(x)}^{\mathsf{T}}\backslash\{0\}:=\{x|x=(x',0),|x'|\le\delta(x)\}\backslash\{0\},
\]
and 
\begin{equation}\label{lower bound of upsi}
\frac{2}{n-2}\ln u_{\psi}(z)\ge\frac{2}{n-2}\ln w^{(x)}(z)\quad \,\,\text{in}\,\,B_{\delta(x)}^{\mathsf{T}}\backslash\{0\}.
\end{equation}
By (\ref{eq:limit at zero}) and Lemma \ref{lem: how far can bar=00003D00007B=00003D00005Clambda=00003D00007D be?},
we have 
\begin{equation}
\lim_{z\rightarrow0}\frac{2}{n-2}\ln u_{\psi}(z)=\frac{2}{n-2}\ln\alpha.\label{eq:limit at zero 2}
\end{equation}

Taking $W:=\frac{2}{n-2}\ln u_{\psi}$, from $A^{u}\in\Gamma_{k}^{+}$,
we know that 
\[
A[W]:=-\nabla_{ij}W+\nabla_{i}W\nabla_{j}W-\frac{|\nabla W|^{2}}{2}\delta_{ij}\in\Gamma_{2}^{+}.
\]
It yields that 
\[
-\sum_{\beta=1}^{n-1}\partial_{\beta\beta}W+\sum_{\beta=1}^{n-1}|\partial_{\beta}W|^{2}-(n-1)\frac{|\nabla W|^{2}}{2}>0\quad\text{on}\,\,\partial\mathbb{R}_{+}^{n}\backslash\{0\},
\]
which is due to $TrA[W]^{\mathsf{\mathrm{T}}}>0$. Thus, we know that
$-\sum_{\beta=1}^{n-1}\partial_{\beta\beta}W>0$ in$\,\,B_{\delta(x)}^{\mathsf{T}}\backslash\{0\}$
for $n\ge3$.

For convenience, we denote $\overline{\nabla}\varphi:=(\partial_{1}\varphi,\cdots,\partial_{n-1}\varphi)$
in the following paragraph. By Lemma \ref{lem:lift lemma}, we know
that there exists a constant vector $\vec{l}\in\mathbb{R}^{n-1}$
such that 
\[
\overline{\nabla}\ln w^{(x)}(0)=\vec{l}\,\,\,\text{for}\,\,\forall\,x\in\partial\mathbb{R}_{+}^{n}.
\]
Otherwise, there exist two points $x_{1}$ and $x_{2}$ such that
$\overline{\nabla}\ln w^{(x_{1})}(0)\neq\overline{\nabla}\ln w^{(x_{2})}(0)$.
Then, by Lemma \ref{lem:lift lemma} and (\ref{eq:zero point equal})(\ref{lower bound of upsi}),
\[
\underset{z\rightarrow0,\,\,z\in\partial\mathbb{R}_{+}^{n}}{\underline{\lim}}W>\frac{2}{n-2}\ln\alpha,
\]
 which contradicts (\ref{eq:limit at zero 2}).

$\vspace{0.5pt}$

Then there exists a constant vector $\vec{V}$ such that $\overline{\nabla}w^{(x)}(0)=\vec{V}\,\text{for}\,\,x\in\partial\mathbb{R}_{+}^{n}.$
For $|y|$ small, 
\begin{align*}
w^{(x)}(y) & =\bar{\lambda}(x)^{n-2}\big(1+(n-2)x\cdot y+O(|y|^{2})\big)u\left(x+\bar{\lambda}(x)^{2}y+O(|y|^{2})\right)\\
 & =\bar{\lambda}(x)^{n-2}(1+(n-2)x\cdot y)u\left(x+\bar{\lambda}(x)^{2}y\right)+O(|y|^{2}),
\end{align*}
and by Lemma \ref{lem: how far can bar=00003D00007B=00003D00005Clambda=00003D00007D be?},
we have 
\begin{align*}
\overline{\nabla}w^{(x)}(0) & =(n-2)\bar{\lambda}(x)^{n-2}u(x)x'+\bar{\lambda}(x)^{n}\overline{\nabla}u(x)\\
 & =(n-2)\alpha x'+\alpha^{n/(n-2)}u(x)^{n/(2-n)}\overline{\nabla}u(x)\\
 & =\vec{V}.
\end{align*}
Therefore, there exists a constant $d$ such that 
\[
\frac{n-2}{2}\alpha^{n/(n-2)}u^{-\frac{2}{n-2}}=-\vec{V}\cdot x'+\frac{1}{2}(n-2)\alpha|x'|^{2}+d\,\,\text{for }x\in\partial\mathbb{R}_{+}^{n}.
\]
Now 
\begin{align*}
u(x',0) & =\bigg(\frac{-\vec{V}\cdot x'+\frac{1}{2}(n-2)\alpha|x'|^{2}+d}{\frac{n-2}{2}\alpha^{n/(n-2)}}\bigg)^{-\frac{n-2}{2}}\\
 & =\alpha\left(\frac{1}{|x'-x_{0}'|^{2}+d_{1}^{2}}\right)^{\frac{n-2}{2}},
\end{align*}
where $u>0$ and $d_{1}$ is a constant.

For simplicity, take $x_{0}'=0$. Then 
\begin{align*}
u(0) & =\alpha d_{1}^{-(n-2)}=\bar{\lambda}(0)^{n-2}u(0)d_{1}^{-(n-2)}
\end{align*}
and now $d_{1}=\bar{\lambda}(0).$

Let $u_{\bar{\lambda}}(y)=\big(\frac{\bar{\lambda}}{|y|}\big)^{n-2}u\big(\frac{\bar{\lambda}^{2}y}{|y|^{2}}\big)$,
where $\bar{\lambda}=\bar{\lambda}(0)$ and 
\begin{align*}
u_{\bar{\lambda}}(x',0) & =\alpha\bigg(\frac{\bar{\lambda}}{|x'|}\bigg)^{n-2}\left(\frac{1}{\bar{\lambda}^{4}/|x'|^{2}+d_{1}^{2}}\right)^{\frac{n-2}{2}}\\
 & =u(x',0)\,\,\text{for}\,\,x\in\partial\mathbb{R}_{+}^{n}.
\end{align*}
Also we know that 
\[
\begin{cases}
\sigma_{k}(A^{u})=2^{k}\binom{n}{k}, \quad g_{u}\in\Gamma_{k}^{+}, \qquad \quad& \text{in}\quad~ \overline{\mathbb{R}_{+}^{n}},\\
\sigma_{k}(A^{u_{\bar{\lambda}}})=2^{k}\binom{n}{k}, \quad g_{u_{\bar{\lambda}}}\in\Gamma_{k}^{+}, & \text{in}\quad~ \overline{\mathbb{R}_{+}^{n}}\backslash\{0\},\\
\mathcal{B}_{k}^{u}=c_{0} & \text{on}\quad \partial\mathbb{R}_{+}^{n},\\
\mathcal{B}_{k}^{u_{\bar{\lambda}}}=c_{0} & \text{on}\quad \partial\mathbb{R}_{+}^{n}\backslash\{0\},\\
u_{\bar{\lambda}}(x',0)=u(x',0) & \text{on}\quad\partial\mathbb{R}_{+}^{n}\backslash\{0\},\\
u-u_{\bar{\lambda}}\ge0 & \text{in}\quad~ \mathbb{R}_{+}^{n}\backslash B_{\bar{\lambda}}.
\end{cases}
\]
Together with \eqref{eq:boundary equation}\eqref{mean curvature expression}\eqref{tangential expression}\eqref{monotone}, we have $\partial_{x_n}u=\partial_{x_n}u_{\bar{\lambda}}$ on $\partial\mathbb{R}_{+}^{n}\backslash\{0\}$.  By the strong maximum principle and the Hopf  lemma, we know that $u=u_{\bar{\lambda}}$
on $\mathbb{R}_{+}^{n}$ and now $u$ satisfies (\ref{eq:strong assmption}).
Then from Theorem \ref{thm:half space liouville-1}, we get Theorem
\ref{thm:half space liouville}. 
\end{proof}

\section{Appendix }

In this appendix, we give a corner Hopf Lemma for second order boundary
condition by mimicking the proof of Li-Zhang \cite{Li-Zhang} and
list several lemmas in Li-Li \cite{Li-Li2} for readers' convenience.

\subsection{Corner Hopf Lemma}

Let $\{A_{ij}(x)\}_{n\times n}$ and $\{C_{\alpha\beta}(x)\}_{(n-1)\times(n-1)}$
be two positive function matrices such that there exist positive constants
$\lambda_{1},\lambda_{2},\varLambda_{1}$,$\varLambda_{2}$ such that
\[
\lambda_{1}\delta_{ij}\le\{A_{ij}\}_{n\times n}\le\varLambda_{1}\delta_{ij},
\]
\[
\lambda_{2}\delta_{\alpha\beta}\le\{C_{\alpha\beta}\}_{(n-1)\times(n-1)}\le\varLambda_{2}\delta_{\alpha\beta}.
\]
For any $\bar{x}\in\partial B_{\lambda}\cap\partial\mathbb{R}_{+}^{n}$,
taking $\Omega=\big(\mathbb{R}_{+}^{n}\backslash B_{\lambda}\big)\cap B_{1}(\overline{x})$,
$\sigma=x_{n}$ and $\rho=|x|^{2}-\lambda^{2}$. Let $\vec{n}$ be
the unit inward normal vector of the surface $\{x_{n}=0\}\cap\partial\Omega$.
The proof of the following theorem is almost same as Li-Zhang \cite{Li-Zhang}
and we omit it. 
\begin{lem}
\label{lem:corner hopf }Let $u\in C^{2}(\overline{\Omega})$ be a
positive function in $\Omega,$ $u(\bar{x})=0$ and there exists a positive
constant $A$ such that 
\[
\begin{cases}
\mathscr{P}u:=-A_{ij}u_{ij}+B_{i}u_{i}\ge-Au & \mathrm{in}\quad\Omega,\\
\mathscr{L}u:=-C_{\alpha\beta}\partial_{\alpha\beta}u+D_{\alpha}\partial_{\alpha}u-C_{0}u_{n}\ge-Au  \qquad& \mathrm{on}\quad\{\sigma=0,\rho>0\},
\end{cases}
\]
where $C_{0}$ is a positive function and $D_{\alpha}$ are functions.
Then 
\[
\frac{\partial u}{\partial\nu'}(\bar x)>0,
\]
where $\nu'$ is the unit normal vector on $\{\sigma=0,\rho=0\}$
entering $\{\sigma=0,\rho>0\}.$ 
\end{lem}

\subsection{Useful Lemmas}

For reader's convenience, we list some classical lemmas in \cite{Li-Li,Li-Li2,Li-Zhang}. 
\begin{lem}[Li-Zhang, J. Anal. Math. 2003]
\label{lem:Cal Lemma-entire space}
Let $f\in C^{1}\left(\mathbb{R}^{n}\right),n\geq1,l>0$. Suppose that
for every $x\in\mathbb{R}^{n}$, there exists $\lambda(x)>0$ such
that 
\[
\left(\frac{\lambda(x)}{|y-x|}\right)^{l}f\left(x+\frac{\lambda(x)^{2}(y-x)}{|y-x|^{2}}\right)=f(y)\quad\mathrm{for}\,\:y\in\mathbb{R}^{n}\backslash\{x\}.
\]
Then for some $a\geq0,d>0,\overline{x}\in\mathbb{R}^{n}$, 
\[
f(x)=\pm\left(\frac{a}{d^{2}+|x-\overline{x}|^{2}}\right)^{\frac{l}{2}}.
\]
\end{lem}

\begin{lem}[Li-Zhang, J. Anal. Math. 2003]
\label{lem:single variable function}
Let $f\in C^{1}(\mathbb{R}_{+}^{n}),n\geq2,\nu>0$. Assume that 
\[
\left(\frac{\lambda}{|y-x|}\right)^{\nu}f\left(x+\frac{\lambda^{2}(y-x)}{|y-x|^{2}}\right)\leq f(y),\quad\forall\lambda>0,\,\,x\in\partial\mathbb{R}_{+}^{n},\,\,|y-x|\geq\lambda,\,\,y\in\mathbb{R}_{+}^{n}.
\]
Then 
\[
f(x)=f(x^{\prime},t)=f(0,t),\quad\forall x=(x^{\prime},t)\in\mathbb{R}_{+}^{n}.
\]
\end{lem}

Lemma \ref{lem:Cal Lemma-entire space} and Lemma \ref{lem:single variable function}
can be found in the appendix of \cite{Li-Zhang}. The following lemma
is about the classification of radial solution, which can be found
in \cite{Li-Li2} for a more general operator including $\sigma_{k}$. 
\begin{thm}[Li-Li, J. Eur. Math. Soc. 2006]
\label{thm:Li-Li-JEMS Thm3} For
$n\geq3$, assume that $u\in C^{2}(\mathbb{B}_{1}^{n})$ is radially
symmetric and satisfies 
\[
\sigma_{k}^{\frac{1}{k}}\left(A^{u}\right)=1,\quad A^{u}\in\Gamma_{k}^{+},\quad u>0\quad\text{ \ensuremath{\mathrm{in}} }\mathbb{B}_{1}^{n}.
\]
Then 
\[
u(x)\equiv\left(\frac{a}{1+b|x|^{2}}\right)^{(n-2)/2}\text{ \ensuremath{\mathrm{in}} }\mathbb{B}_{1}^{n},
\]
where $a>0,b\geq-1$ and $\sigma_{k}^{\frac{1}{k}}\big((2b/a^{2})\mathbb{I}_{n\times n}\big)=1$. 
\end{thm}

\bigskip

\noindent\textbf{Conflict of interest:} The author states that
there is no conflict of interest.

\end{document}